\documentclass[10pt]{amsart}

\usepackage{amsmath, amssymb}
\usepackage{graphicx}
\usepackage{subfigure}
 \usepackage{mathrsfs}
\newcommand{\no}[1]{#1}
\renewcommand{\no}[1]{}
\no{\usepackage{times}\usepackage[subscriptcorrection,
 slantedGreek, nofontinfo]{mtpro}
\renewcommand{\Delta}{\upDelta}}
\usepackage{color}

\usepackage{esint}
\usepackage[svgnames]{xcolor} 
\usepackage[colorlinks,citecolor=red,pagebackref,hypertexnames=false,breaklinks]{hyperref}
\usepackage{pgf,tikz}
\usepackage{pdfsync}

\usepackage[bottom]{footmisc}
\usepackage{dsfont}

\usepackage[T1]{fontenc}
\usepackage{lmodern}
\usepackage{mathtools}  
\usepackage{amssymb}
\usepackage{lipsum}
\usepackage{mathrsfs}
\usepackage{color}
\usepackage{bbm}
\usepackage{enumerate}

\date{\today}
\newtheorem{lemma}{Lemma}[section]

\newtheorem{proposition}{Proposition}[section]
\newtheorem{theorem}{Theorem}[section]

\def\N{\mathbb{N}}
\def\R{\mathbb{R}}
\def\C{\mathbb{C}}
\def\Z{\mathbb{Z}}
\def\im{\mathfrak I}
\def\re{\mathfrak R}

\def\hat{\widehat}
\def\tilde{\widetilde}

\def\cH{{\mathcal H}}

\newcommand{\norm}[1]{\|#1\|}
\newcommand{\abs}[1]{\left|#1\right|}

\newcommand{\be}{\begin{equation}}
\newcommand{\ee}{\end{equation}}
\newcommand{\ba}{\begin{array}}
\newcommand{\ea}{\end{array}}

\newcommand{\bea}{\begin{eqnarray*}}
\newcommand{\eea}{\end{eqnarray*}}
\newcommand{\bean}{\begin{eqnarray}}
\newcommand{\eean}{\end{eqnarray}}

\def\cydot{\leavevmode\raise.4ex\hbox{.}}
\newcommand{\normmm}[1]{{\left\vert\kern-0.25ex\left\vert\kern-0.25ex\left\vert #1
    \right\vert\kern-0.25ex\right\vert\kern-0.25ex\right\vert}}

\newcommand\gobblepars{%
    \@ifnextchar\par%
        {\expandafter\gobblepars\@gobble}%
        {}}

\title[inverse  source problem]{Logarithmic stable recovery of the source and the initial state of time fractional diffusion equations}

\author{Yavar Kian}
\address{Y.Kian, Aix-Marseille Univ, Universit\'e de Toulon, CNRS, CPT, Marseille, France}
\email{yavar.kian@univ-amu.fr}
\author{\'Eric Soccorsi}
\address{\'E. Soccorsi, Aix-Marseille Univ, Universit\'e de Toulon, CNRS, CPT, Marseille, France}
\email{eric.soccorsi@univ-amu.fr}
\author{Faouzi Triki}
\address{F. Triki, Laboratoire Jean Kuntzmann, UMR CNRS 5224, Universit\'e Grenoble-Alpes, 700 Avenue  Centrale,
38401 Saint-Martin-d'H\`eres, France}
\email{faouzi.triki@univ-grenoble-alpes.fr}
\thanks{This work is partially supported the Agence Nationale la Recherche (ANR) under grant ANR-17-CE40-0029 (projet MultiOnde).}

\begin{document}

\begin{abstract}
In this paper we study the inverse problem of identifying a source or an initial state in a time-fractional diffusion
equation from the knowledge of a single boundary measurement. We derive logarithmic stability estimates for both inversions.  
These results show that the ill-posedness increases exponentially when the fractional derivative order tends to zero,
while it exponentially decreases when the regularity of the source or the initial state becomes larger.  The stability estimate concerning the problem of recovering the initial state can be considered as a weak observability inequality in control theory. 
The analysis is  mainly based  on Laplace inversion techniques and  a precise quantification of the  unique continuation property  for the  resolvent  of the time-fractional diffusion operator as a function of the frequency in the complex plane. We also determine a global time regularity for the  time-fractional diffusion equation  which is of interest itself. 
 
\end{abstract}

\maketitle

\section{Introduction and main results}

\subsection{Settings}
Let $\Omega \subset \R^d, d=2, 3,$ be a bounded domain containing the origin, with $C^2$ boundary $\partial \Omega$. 
With reference to \cite{P}, the
Riemann-Liouville integral operator of order  $\beta$, denoted by $I^\beta$, is defined by
$$
I^\beta h(t,\cdot):=\frac{1}{\Gamma(\beta)}\int_0^t\frac{h(\tau,\cdot)}{(t-\tau)^{1-\beta}}\ d\tau,
$$
and the Riemann-Liouville fractional derivative of order $\beta$ is $D_t^\beta:= \partial_t\circ I^{1-\beta}$. Set 
$$\partial_t^\beta h:=D_t^\beta (h-h(0,\cdot)),\ h\in  C([0,+\infty);L^2(\Omega)).$$
The operator $\partial_t^\beta$ is called the Caputo fractional derivative of order $\beta$, as we have
$$\partial_t^\beta h =I^{1-\beta}\partial_t h,\ h\in W^{1,1}_{loc}(\R_+;L^2(\Omega)), $$
where $\R_+:=(0,+\infty)$.

For $\alpha \in (0, 1)$ fixed, we consider the following initial boundary value problem (IBVP)
\bean 
\label{directproblem}
\left\{ \ba{lllccc}
\partial^\alpha_t u(t, x) -\Delta u(t, x) = g(t)f(x), & (t,x)\in   \R_+\times \Omega, \\
u(t, x) = u_0(x), & x\in  \Omega,
\\
u(t, x) = 0, & (t,x)\in  (0, +\infty) \times \partial \Omega,
\ea
\right.
\eean
where $f \in L^2(\Omega)$, $g \in 
L^\infty(\R_+)$ and $u_0 \in H^1_0(\Omega)$. In this article, assuming that the function $g$ is known and satisfies an appropriate condition that we will make precise further, we aim to study the stability issue in the inverse problem of determining either the source term $f$ or the initial state $u_0$, from a single boundary measurement $\partial_\nu u = \nabla u \cdot \nu $ on $\R_+ \times \partial \Omega$, of the solution $u$ to \eqref{directproblem}. Here and in the remaining part of this text, we denote by $\nu$ the outward unit normal vector to $\partial \Omega$.

The time-fractional diffusion system \eqref{directproblem} describes anomalous diffusion in homogeneous media. It has multiple engineering applications in geophysics, environmental science and biology, see e.g., \cite{AG,CSLG,JR}. From a mathematical viewpoint, the time-fractional diffusion equation of \eqref{directproblem} can be seen as a corresponding macroscopic model to microscopic diffusion processes governed by a continuous-time random walk, see e.g., \cite{BG,MK}. 

\subsection{A brief review of the existing literature}
\label{sec-review}

Inverse source problems have received a lot of attention from the mathematical community over the last decade, owing it to the  major impact they made in many areas, including medical diagnosis and industrial nondestructive testing. We refer the reader to \cite{Is, BLT} for an overview of inverse source problems for partial differential equations, and to \cite{JR,LLY} for the study of these problems in the framework of fractional diffusion equations. 

While several authors have already addressed the inverse problem of retrieving the space-varying part of the source term in a fractional diffusion equation, see e.g., \cite{FK,SY}, only uniqueness results are available in the mathematical literature, see e.g., \cite{JLLY,JR,KLY,KSXY,RZ} (see also \cite{KR1} for some related inverse problem) with the exception of the recent stability  result of \cite{CL} stated with specific norms. This is not surprising since the classical methods used to build a stability estimate for the source of parabolic ($\alpha=1$) or hyperbolic ($\alpha=2$) systems, see e.g., \cite{CY,IY,KSS, ACT}, do not apply in a straightforward way to time-fractional diffusion equations. 

As far as we know, the only mathematical work dealing with the stability issue of the inverse source problem under consideration in the present article, can be found in \cite{CL,JKZ,KY2}. While the authors of \cite{JKZ,KY2} studied this problem in the peculiar framework of a cylindrical domain $\Omega'\times(-\ell,\ell)$, where $\Omega'$ is an open subset of $\R^{d-1}$ and $\ell \in\R_+$, the Lipschitz stability estimate presented in \cite{CL} is derived for $\alpha \in (1,2)$ under a specifically designed topology induced by the adjoint system of the fractional wave equation. 

However, 
we could not find such thing as a stability inequality with respect to the usual norm in $L^2(\Omega)$, of the space-dependent part of the source term of time-fractional diffusion equations posed in a general bounded domain, in the mathematical literature. Besides, the main achievement of this article is the logarithmic-stable determination through one Neumann data, of either the space-varying part of the source term or the initial state of the IBVP \eqref{directproblem}. The corresponding stability estimates are given in Theorems \ref{thm-inv1} and \ref{thm-inv2} below, but prior to stating these two inverse results, we turn our attention to the direct problem associated with \eqref{directproblem}. 

\subsection{The direct problem}
\label{sec-fwd}
With reference to \cite{EK,KLY,KY1}, we define a weak solution to the IBVP as a function $u \in L_{loc}^1(\R_+;L^2(\Omega))$ satisfying the three following conditions simultaneously:
\begin{enumerate}[i)]
\item $D^\alpha_t [u-u_0](t, x) -\Delta u(t, x) = f(x)g(t)$ in the distributional sense in $\R_+\times\Omega$;
\item $I^{1-\alpha} u\in W_{loc}^{1,1}(\R_+;H^{-2}(\Omega))$ and $I^{1-\alpha} [u-u_0](0, x)=0$ for a.e. $x\in\Omega$;
\item $p_0:=\inf\{\tau>0:\ e^{-\tau t}u\in L^1(\R_+;L^2(\Omega))\}<\infty$ and
$$ \exists p_1\ge p_0,\ \forall p\in\C,\  \re p >p_1 \Longrightarrow \tilde{u}(p,\cdot):=\int_0^{+\infty} e^{-p t}u(t,\cdot)\ d t\in H^1_0(\Omega). $$
\end{enumerate}
Here and the remaining part of this article, $\re z$ (resp., $\im z$) denotes the real part (resp., the imaginary part) of the complex number $z$.
Notice from iii) that $\tilde{u}$ is the Laplace transform in time of $u$ with respect to the time-variable $t$.

As long as finite time evolution is concerned, we can rely on the results of \cite{KLY,KY1,SY}. Indeed, they ensure us for all $\alpha \in (0,1)$, all $T>0$ and all $(f,g,u_0)  \in L^2(\Omega) \times 
L^\infty(\R_+) \times \left(  H^1_0(\Omega)\cap H^2(\Omega) \right)$ that the IBVP \eqref{directproblem} admits a unique weak solution  
$u \in L^2_{\mathrm{loc}}(\R_+; H^2(\Omega)\cap H^1_0(\Omega))$ satisfying
$$ \norm{u}_{ L^2(0, T; H^2(\Omega))} +\norm{\partial_t^\alpha u}_{L^2((0,T)\times\Omega)} \leq 
 C\left( \norm{g}_{L^\infty(0, T)} \norm{f}_{L^2(\Omega)} + \norm{u_0}_{H^k(\Omega)} \right),$$
for some constant $C >0$ depending only on $\alpha$, $\Omega$ and $T$. However, since $C$ may blow up as $T$ tends to infinity, it is not clear how the global time properties of the solution on $\R_+$, needed by the analysis of the inverse problem under study in this article, can be inferred from the above estimate. For this reason we proceed by establishing the following global time existence and uniqueness result.
 
 \begin{proposition}
 \label{pr-fwd}
Let $\alpha \in (0, 1)$ and let $(f,g,u_0) \in L^2(\Omega) \times \left( L^1(\R_+)\cap L^\infty(\R_+) \right) \times \left( H^1_0(\Omega)\cap H^2(\Omega) \right)$. Then, there exists a unique weak solution $u \in L^r(\R_+; H^{\frac{7}{4}}(\Omega))$, $r>\alpha^{-1}$, to the IBVP \eqref{directproblem}. Moreover, we have 
 \bean \label{reg.direct}
 \|u\|_{L^r(\R_+; H^{\frac{7}{4}}(\Omega))} \leq 
 C\left(\|g\|_{L^r(\R_+)}\|f\|_{L^2(\Omega)} + \|u_0\|_{H^2(\Omega)} \right),
 \eean
for some positive constant $C$ depending only on $\alpha$ and $\Omega$.
 \end{proposition}


Having expressed Proposition \ref{pr-fwd}, whose proof can be found in Appendix \ref{sec-prfwd}, we turn now to stating the two main results of this article. Each of them is concerned with one of the inverse problems described in the introduction.  
\subsection{Inverse problems}

In this section we present two stability estimates. The first one is associated with the determination of the source term $f$ of the IBVP \eqref{directproblem}, by one boundary measurement of the solution over the entire time-span $\R_+$. The second one is related to the same inverse problem where the unknown is replaced by the initial state $u_0$. Prior to stating these two inequalities, we introduce the following notations. Firstly, for all $\tau\in\R$, we note $\lfloor \tau\rfloor$ the integer part of  $\tau$. Secondly, for all $k\in\mathbb N:=\{1,2,\ldots\}$, we denote by $H^k_0(\Omega)$ the closure of $C^\infty_0(\Omega)$ in the topology of the $k$-th order Sobolev space $H^k(\Omega)$. 

We start with the determination of the source term $f$.

 \begin{theorem}
 \label{thm-inv1}
 Let $\alpha \in (0,1)$, let $f \in H^k_0(\Omega)$ for some $k\in \N$ and let $g \in L^\infty(\R_+)\cap L^1(\R_+)$ satisfy the condition
 \begin{equation}
 \label{gcond}
 \exists c_0>0,\ \forall p \in \R_+,\ \abs{\tilde{g}(p)} \geq c_0.
 \end{equation} 
 Assume that $u_0=0$ and denote by $u$ the $L^{\frac{2}{\alpha}}(\R_+; H^{\frac{7}{4}}(\Omega))$-solution to \eqref{directproblem}, and let  $s= 1+ \lfloor \frac{2}{\alpha}\rfloor $. \\
 Then, for all $\theta \in (0,1)$, there exists $\varepsilon_0= \varepsilon_0\left(\Omega, d, k, \theta, \frac{\norm{f}_{H^k(\Omega)}}{\|f \|_{L^1(\Omega)}}\right)\in (0, 1)$ such that we have
\begin{equation}
\label{main1}
\norm{f}_{L^2(\Omega)}\leq 
C \norm{f}_{H^k(\Omega)} 
 \left\{ \begin{array}{cl}  \abs{\ln \norm{\partial_\nu u}_{L^{\frac{2}{\alpha}}(\R_+;L^2(\partial \Omega))}}^{-\frac{k}{1+\theta}} & \mbox{if}\ s \leq d +\frac{d^2}{2k}\\   \abs{\ln \norm{\partial_\nu u}_{L^{\frac{2}{\alpha}}(\R_+;L^2(\partial \Omega))}}^{-\frac{d^2}{2(s-d)(1+\theta)}}  & \mbox{if}\ s > d +\frac{d^2}{2k}, \end{array} \right. 
\end{equation}
whenever $\norm{\partial_\nu u}_{L^{\frac{2}{\alpha}}(\R_+;L^2(\partial \Omega))} \in (0, \varepsilon_0)$. Here, $C$ is a positive constant depending only on $\Omega$, $d$, $k$, $\theta$ and $c_0$.

\end{theorem}
 
The corresponding statement for the determination of the initial state $u_0$, is as follows.

\begin{theorem}
\label{thm-inv2}
 Let $\alpha \in (0,1)$ and let $u_0 \in H^k_0(\Omega)$, where $k\in \N$.
 Assume that $f=0$, denote by $u$ the $L^{\frac{2}{\alpha}}(\R_+; H^{\frac{7}{4}}(\Omega))$-solution to the IBVP \eqref{directproblem}, and  let  $s= 1+ \lfloor \frac{2}{\alpha}\rfloor $. \\
 Then, for all $\theta \in (0,1)$, there exists $\varepsilon_0= \varepsilon_0\left(\Omega, d, k, \theta, \frac{\norm{u_0}_{H^k(\Omega)}}{\norm{u_0}_{L^1(\Omega)}}\right)\in (0, 1)$ such that we have 
 \bean 
 \label{main2}
 \norm{u_0}_{L^2(\Omega)}\leq 
C \norm{u_0}_{H^k(\Omega)}  \left\{ \begin{array}{cl}  \abs{\ln \norm{\partial_\nu u}_{L^{\frac{2}{\alpha}}(\R_+;L^2(\partial \Omega))}}^{-\frac{k}{1+\theta}} & \mbox{if}\ s \leq d +\frac{d^2}{2k}\\   \abs{\ln \norm{\partial_\nu u}_{L^{\frac{2}{\alpha}}(\R_+;L^2(\partial \Omega))}}^{-\frac{d^2}{2(s-d)(1+\theta)}}  & \mbox{if}\ s > d +\frac{d^2}{2k},\end{array} \right. 
\eean
provided $\norm{\partial_\nu u}_{L^{\frac{2}{\alpha}}(\R_+;L^2(\partial \Omega))} \in (0, \varepsilon_0)$. Here, $C$ is a positive constant  depending only on $\Omega$, $d$, $k$ and $\theta$.
\end{theorem}

To our knowledge, Theorem \ref{thm-inv1} (resp., Theorem \ref{thm-inv2}) is the only existing mathematical result on the stable recovery of the space-varying part of the source term (resp., the initial state) of a fractional diffusion equation posed in a general spatial domain. As already mentioned in Section \ref{sec-review}, there are two comparable results available in \cite{JKZ,KY2} but they only apply to cylindrical shaped domains.\\

Notice that the stability of the recovery of either $f$ or $u_0$, expressed in \eqref{main1} and \eqref{main2}, respectively, degenerates exponentially fast as the fractional order $\alpha$ tends to zero. In contrast, this stability exponentially 
 improves as their Sobolev regularity order $k$ increases. \\
 
 The results of Theorem \ref{thm-inv2} can also be seen as a weak observability inequality  that may  lead to an approximate controllability for the time-fractional diffusion equation in control theory \cite{Zu, AT, ACT}.  This problem  is of importance in
 Photoacoustic  imaging in the sub-diffusion regime \cite{ABJW, SC, RT}.

\subsection{Outline}
This article is organized as follows. Section 2 contains the proof of the two stability estimates of  \eqref{main1}-\eqref{main2}. In the Appendix A, we examine the direct problem associated with the IBVP \eqref{directproblem} and give the proof of Proposition \ref{pr-fwd}. In Appendix B, we establish a unique continuation result, which is a cornerstone in the proof of
Theorems \ref{thm-inv1} and  \ref{thm-inv2}.

\section{Proof of Theorems \ref{thm-inv1} and \ref{thm-inv2}}
In this section we derive the two stability inequalities \eqref{main1} and \eqref{main2}. Our strategy is to study the IBVP \eqref{directproblem} in the so-called frequency domain, that is to say that we examine the elliptic system derived from \eqref{directproblem} upon applying the Laplace transform with respect to the time-variable $t$. This is made precise in the coming section, which is a preamble to the proof of Theorems \ref{thm-inv1} and \ref{thm-inv2}. 

\subsection{Preliminaries}
Let $u$ be the weak solution to \eqref{directproblem} given by Propisition \ref{pr-fwd}. In light of iii) in Section \ref{sec-fwd}, the Laplace transform in time $\tilde{u}(p,\cdot)$ of $u$ is well-defined provided the real part of $p \in \C$ is sufficiently large. As a matter of fact, it can be checked from 
\cite[Theorem 1.3]{K} or \cite[Theorem 4.1]{KLY} that $\tilde{u}(p,\cdot)$ is well-defined for all $p\in \C_+:= \{z\in \mathbb C,\ \re(z)>0 \}$ and solves the following boundary value problem (BVP)
\bean \label{BVP}
 \left\{ \ba{lllccc}
(-\Delta+p^\alpha) \tilde{u}(p,\cdot)=\tilde{g}(p) f+p^{\alpha-1}u_0 & \mbox{in}\ \Omega,\\
\tilde{u}(p,\cdot) = 0 & \mbox{on}\ \partial \Omega.
\ea
\right.
\eean
Notice that since $g\in L^\infty(\R_+)$, $\tilde{g}(p,\cdot)$ is well-defined for $p \in \C_+$ as well. 

Next, for all $z \in \C^* \setminus i \R_+:=\{\tau \in \C,\ -i\tau\notin[0,+\infty)\}$ and for all $q \in \R$, we set $z^q := e^{q \log z}$, where $\log$ denotes the complex logarithm function defined and holomorphic on $\C^* \setminus i \R_+$. Thus, for all $\omega \in \R_+$,
$U(\omega,\cdot):=\tilde{u}(\omega^{\frac{2}{\alpha}},\cdot)$ is a solution to
\bean 
\label{ellipticequation}
\left\{ \ba{lllccc} (-\Delta + \omega^2) U= \tilde{g}(\omega^{\frac{2}{\alpha}}) f+\omega^{2-\frac{2}{\alpha}}
u_0 & \mbox{in}\ \Omega,\\
U= 0 &\mbox{on}\ \partial \Omega.
\ea
\right.
\eean
Let us denote by $(\lambda_k)_{k\geq 1} \in \R_+^{\N}$ the eigenvalues of the self-adjoint operator $A$ in $L^2(\Omega)$, acting as $-\Delta$ on its domain $H^1_0(\Omega)\cap H^2(\Omega)$, which is positive and has a compact resolvent. Since $F:=\tilde{g}(\omega^{\frac{2}{\alpha}}) f+\omega^{2-\frac{2}{\alpha}} u_0 \in L^2(\Omega)$, the BVP \eqref{ellipticequation} admits a unique solution $U=(A-\omega^2)^{-1} F$ within the space
$H^1_0(\Omega)\cap H^2(\Omega)$, provided we have $\omega \in \C\setminus \{ \pm i\lambda_k,\ k\geq 1\}$. In particular, this entails that $U(\omega,\cdot) \in H^2(\Omega)$ for all $\omega \in \C_+$.

Further, recalling that the Fourier transform $\hat{\phi}$ of a function $\phi \in L^1(\Omega)$ reads
\bean
\label{def-Fourier}
\hat{\phi}(\xi) = \frac{1}{(2\pi)^{\frac{d}{2}}}\int_\Omega e^{-ix\cdot \xi} \phi(x) dx,\ \xi\in\R^d,
\eean
we pick $\xi \in \mathbb S^{d-1}$, multiplying the first equation of \eqref{ellipticequation} by $e^{\omega x\cdot \xi }$, 
where 
$\omega \in \R_+$ is fixed, and we integrate by parts over $\Omega$. We obtain that
\bean \label{laplcasource}
\tilde{g}(\omega^{\frac{1}{2\alpha}}) \hat{f}(i\omega \xi)+\omega^{2-\frac{2}{\alpha}} \hat{u_0}(i \omega \xi)
= (2\pi)^{-\frac{d}{2}}\int_{\partial \Omega}\partial_\nu U(\omega, x) e^{\omega x\cdot \xi } d\sigma(x),\ \xi \in \mathbb S^{d-1},\ \omega \in \R_+.
\eean 
The above identity is a stepping stone to the proof of Theorems \ref{thm-inv1} and \ref{thm-inv2}.

\subsection{Proof of Theorem \ref{thm-inv1}}
\label{sec-pr_inv1}
Since $u_0=0$ by assumption and $\abs{\tilde{g}(p)} \geq c_0>0$ for all $p \in \R_+$, according to \eqref{gcond}, it follows from
\eqref{laplcasource} that
\begin{equation}
\label{eq0}
\abs{\hat{f}(i\omega \xi)} \leq c_0^{-1} (2\pi)^{-\frac{d}{2}}  \abs{\partial \Omega}^{\frac{1}{2}} 
e^{\kappa_\Omega \omega } \norm{\partial_\nu U(\omega,\cdot)}_{L^2(\partial \Omega)},\ \xi \in \mathbb S^{d-1},\ \omega \in \R_+,
\end{equation}
where $\kappa_\Omega := \sup_{x\in \partial \Omega} \abs{x}$.\\
Next, we have $u\in L^{\frac{2}{\alpha}}(\R_+; H^{\frac{7}{4}}(\Omega))$ from Proposition \ref{pr-fwd}, and hence
$\partial_\nu u\in L^{\frac{2}{\alpha}}(\R_+; H^{\frac{1}{4}}(\partial\Omega))\subset L^{\frac{2}{\alpha}}(\R_+; L^2(\partial\Omega))$, and
$$\partial_\nu U(\omega, x)=\tilde{\partial_\nu u}(\omega^{\frac{2}{\alpha}},x),\ x\in\partial\Omega,\ \omega \in \R_+,$$
from \cite[Step 2 in the proof of Theorem 2.2]{KLLY}. Thus, applying H\"older's inequality, we get that
\begin{eqnarray}
\norm{\partial_\nu U(\omega,\cdot)}_{L^2(\partial \Omega)} &\leq & \int_0^{+\infty}e^{-\omega^{\frac{2}{\alpha}} t} \norm{\partial_\nu u(t,\cdot)}_{L^2(\partial \Omega)}dt \nonumber \\
&\leq & \left( \int_0^{+\infty} e^{-\frac{2}{2-\alpha} \omega^{\frac{2}{\alpha}} t}dt\right)^{\frac{2-\alpha}{2}} \norm{\partial_\nu u}_{L^{\frac{2}{\alpha}}(\R_+;L^2(\partial \Omega))} \nonumber \\
&\leq & \omega^{\frac{\alpha-2}{\alpha}} \norm{\partial_\nu u}_{L^{\frac{2}{\alpha}}(\R_+;L^2(\partial \Omega))}. \label{eq1}
\end{eqnarray} 
Now, putting \eqref{eq0} together with \eqref{eq1}, we find for all $\xi \in \mathbb S^{d-1}$ that 
$$
\abs{\hat{f}(i\omega \xi)} \leq  c_0^{-1} (2\pi)^{-\frac{d}{2}}  \abs{\partial \Omega}^{\frac{1}{2}} \omega^{\frac{\alpha-2}{\alpha}} e^{\kappa_\Omega \omega } \norm{\partial_\nu u}_{L^{\frac{2}{\alpha}}(\R_+;L^2(\partial \Omega))},\ \omega \in \R_+.
$$
Since $\omega^{\frac{\alpha-2}{\alpha}} \le \left( \frac{2+\omega}{\omega} \right)^s$ for all $\omega \in \R_+$, where $s:= 1+ \lfloor \frac{2}{\alpha}\rfloor$, we deduce from the above line that
\bean
\label{a1}
\left( \frac{\omega}{2+\omega}\right)^s e^{-\kappa_\Omega \omega } \abs{\hat{f}(i\omega \xi)} \leq  c_0^{-1} (2\pi)^{-\frac{d}{2}}  \abs{\partial \Omega}^{\frac{1}{2}}  \norm{\partial_\nu u}_{L^{\frac{2}{\alpha}}(\R_+;L^2(\partial \Omega))},\ \omega \in [0,+\infty).
\eean
Further, the right-hand-side on \eqref{a1} being independent of $\xi$, we substitute $-\xi$ for $\xi$ in \eqref{a1} and obtain that
$$
\left( \frac{-\omega}{2-\omega}\right)^s e^{\kappa_\Omega \omega} \abs{\hat{f}(i\omega \xi)} \leq  c_0^{-1} (2\pi)^{-\frac{d}{2}}  \abs{\partial \Omega}^{\frac{1}{2}}  \norm{\partial_\nu u}_{L^{\frac{2}{\alpha}}(\R_+;L^2(\partial \Omega))},\ \omega 
\in (-\infty,0].
$$
Then we use the fact that $e^{\kappa_\Omega \omega}\ge e^{-2 \kappa_\Omega} e^{-\kappa_\Omega \omega}$ 
and $(2-\omega)^{-1} \geq 3^{-1} (2+\omega)^{-1}$ whenever $\omega \in [-1,0)$, to get that
$
\left( \frac{-\omega}{2+\omega}\right)^s e^{-\kappa_\Omega \omega} \abs{\hat{f}(i\omega \xi)} \leq  3^s c_0^{-1} (2\pi)^{-\frac{d}{2}}  e^{2 \kappa_\Omega} \abs{\partial \Omega}^{\frac{1}{2}}  \norm{\partial_\nu u}_{L^{\frac{2}{\alpha}}(\R_+;L^2(\partial \Omega))}$ for all $\omega \in [-1,0)$.
From this and \eqref{a1} it then follows for all $\xi \in \mathbb S^{d-1}$ that
\begin{equation}
\label{a2}
\abs{\left( \frac{\omega}{2+\omega}\right)^s e^{-\kappa_\Omega \omega } \hat{f}(i\omega \xi)} \leq  \epsilon,\ \omega \in [-1,+\infty),
\end{equation}
where
\begin{equation}
\label{a3}
\epsilon:=3^s c_0^{-1} (2\pi)^{-\frac{d}{2}} e^{2 \kappa_\Omega} \abs{\partial \Omega}^{\frac{1}{2}} \norm{\partial_\nu u}_{L^{\frac{2}{\alpha}}(\R_+;L^2(\partial \Omega))}.
\end{equation}
Having seen this, we put $Q:=\{ z \in \C,\ \re z >-1\ \mbox{and}\ \im z<0 \}$, $F(z) :=   \left(\frac{z}{2+z}\right)^{s} e^{-\kappa_\Omega z} \hat{f}(iz\xi)$
for all $z \in \C$ such that $\re z \ge -1$, and we apply Theorem \ref{thmUC}. We obtain that
\begin{equation}
\label{a4}
\abs{F(z)} \leq M  m^{w(z)},\ z\in \overline{Q} \setminus \{ -1 \},
\end{equation}
where $M:=1+(2\pi)^{\frac{d}{2}} e^{2 \kappa_\Omega} \norm{f}_{L^1(\Omega)}$, $m:=\sup_{t\in[-1,+\infty)} \abs{F(t)}$
and $w(z):=\frac{2}{\pi} \left( \frac{\pi}{2}+ \arg(z+1) \right)$.
Here and below, $\arg z$, for all $z \in \C$ satisfying $\re z>0$, denotes the angle $\theta \in \left(-\frac{\pi}{2},\frac{\pi}{2} \right)$ such that $z=\abs{z}e^{i \theta}$. Further, since $m \le \epsilon$ from \eqref{a2}, and $m(z) \ge 0$ for all $z \in \overline{Q} \setminus \{ - 1 \}$, \eqref{a4} yields that
$$
\abs{\left(\frac{z}{2+z}\right)^{s} e^{-\kappa_\Omega z} \hat{f}(iz\xi)} \leq M \epsilon^{w(z)},\ z \in \overline{Q} \setminus \{ - 1 \}.
$$
In the particular case where $z=-it$, $t \in \R_+$, this leads to 
$\frac{t^s}{(4+t^2)^{\frac{s}{2}}} \abs{\hat{f}(t\xi)} \leq M \epsilon^{w(-it)}$ and
$w(-it) =  \frac{2}{\pi} \left(\frac{\pi}{2}+\arg(1-it)\right)=\frac{2}{\pi} \left(\frac{\pi}{2}-\arctan t \right)=\frac{2}{\pi} \arctan t^{-1}$.
As a consequence we have
\begin{equation}
\label{a5}
\abs{\hat{f}(t\xi)} \leq  M \frac{(4+t^2)^{\frac{s}{2}}}{t^s} \epsilon^{\frac{2}{\pi}\arctan t^{-1}},\ t \in \R_+,\ \xi \in \mathbb S^{d-1}.
\end{equation}
Therefore, for all $\delta \in (0,1)$ and all $R \in (1,+\infty)$, we get
\begin{equation}
\label{ff2}
\norm{\hat{f}}_{L^2(B_{R,\delta}(0))}
\leq  h_s(\delta,R), 
\end{equation}
through standard computations, where
\begin{equation}
\label{a6}
h_s(\delta,R):= 5^{\frac{s}{2}} C_d \epsilon^{\frac{2}{\pi} \arctan R^{-1}} k_s(\delta,R)\ \mbox{and}\ k_s(\delta,R):=\left\{ \begin{array}{ll} R^d & \mbox{if}\ s < d,\\
-\ln \delta +R^d& \mbox{if}\ s=d,\\
\delta^{-(s-d)} +R^d& \mbox{if}\ s > d,
\end{array}\right.
\end{equation}
and $C_d$ is a positive constant depending only on $d$.
Here and in the remaining part of this article, $B_\rho(0)$ (resp., $\overline{B_\rho(0)}$), $\rho \in \R_+$, 
denotes the open (resp., closed) 
ball of $\R^d$ centered at $0$ and with radius $\rho$, i.e., $B_\rho(0):=\{ x \in \R^d,\ \abs{x}<\rho \}$ (resp., $\overline{B_\rho(0)}:=\{ x \in \R^d,\ \abs{x} \leq\rho \}$), and $B_{R,\delta}(0):= B_R(0) \setminus \overline{B_\delta(0)}$. Notice that the upper-bound on $\norm{\hat{f}}_{L^2(B_{R,\delta}(0))}$ given by \eqref{a6} was obtained upon decomposing the annulus $B_{R,\delta}(0)$ into $B_{R,1}(0) \cup B_{1,\delta}(0)$ and
majorizing $\abs{\hat{f}(t\xi)}$ by $5^{\frac{s}{2}} M \epsilon^{\frac{2}{\pi} \arctan R^{-1}}$ for all $t \in B_{R,1}(0)$, and by
$5^{\frac{s}{2}} M t^{-s} \epsilon^{\frac{2}{\pi}\arctan R^{-1}}$ for all $t \in B_{1,\delta}(0)$.

Further, since $f\in L^1( \Omega)$, we have $\abs{\hat{f}(\xi)} \leq (2\pi)^{-\frac{d}{2}} \norm{f}_{L^1(\Omega)}$ for all $\xi \in \R^d$ and hence $\abs{\hat{f}(\xi)} \leq (2\pi)^{-\frac{d}{2}} \abs{ \Omega}^{\frac{1}{2}} \norm{f}_{H^k(\Omega)}$ by the Cauchy-Schwarz inequality. This entails that
\begin{equation}
\label{a7}
\norm{\hat{f}}_{L^2( B_\delta(0))} \leq C \delta^\frac{d}{2} \norm{f}_{H^k(\Omega)},
\end{equation}
where, here and below, $C$ denotes a generic positive constant depending only on $d$ and $\Omega$, which may change from line to line.
Next, since the function $f \in H^k_0(\Omega)$ extended by zero in $\R^d \setminus \Omega$, lies in $H^k(\R^d)$, we infer from the
Fourier-Plancherel theorem that
\begin{eqnarray*}
\norm{\hat{f}}_{L^2(\R^d \setminus B_R(0))}^2 & = & \int_{\R^d\setminus B_R(0)} \abs{\hat{f}(\xi)}^2 d \xi\\
&\leq & R^{-2k} \int_{\R^d \setminus B_R(0)} (1+\abs{\xi}^2)^k \abs{\hat{f}(\xi)}^2 d\xi\\
&\leq & R^{-2k} \int_{\R^d} (1+ \abs{\xi}^2)^k \abs{\hat{f}(\xi)}^2 d\xi\\
& \leq & R^{-2k} \norm{f}^2_{H^k(\Omega)}.
\end{eqnarray*}
From this and \eqref{a7} it then follows that $\norm{\hat{f}}_{L^2(\R^d \setminus B_{R,\delta}(0))} \leq 
C(\delta^ \frac{d}{2} +R^{-k}) \norm{f}_{H^k(\Omega)}$, which together with \eqref{ff2} yields
$\norm{\hat{f}}_{L^2(\R^d)} \leq h_s(\delta,R)+ C  (\delta^ \frac{d}{2} +R^{-k}) \norm{f}_{H^k(\Omega)}$. And since $\norm{f}_{L^2(\Omega)}  = \norm{\hat{f}}_{L^2(\R^d)}$ from the Fourier-Plancherel theorem, we have
\begin{equation}
\label{a7.1}
\norm{f}_{L^2(\Omega)} \leq h_s(\delta,R)+ C  (\delta^ \frac{d}{2} +R^{-k}) \norm{f}_{H^k(\Omega)}.
\end{equation}
The rest of the proof is to estimate $h_s(\delta,R)$ with the aid of \eqref{a6}. This leads us to examine the three cases $s<d$, $s=d$ and $s>d$ separately.\\
 
\noindent {\it First case: $s<d$}. 
Taking $\delta = R^{-\frac{2k}{d}}$ in \eqref{a7.1}, we find that
\begin{equation}
\label{a8}
\norm{\hat{f}}_{L^2(\R^d)}\leq 5^{\frac{s}{2}} C_d M R^d \epsilon^{\frac{2}{\pi} \arctan R^{-1}}+ C R^{-k} \norm{f}_{H^k(\Omega)}.
\end{equation}
Next, assuming without loss of generality that $\epsilon \in (0,1)$, we pick $\theta\in(0,1)$, choose
$R=(-\ln \epsilon)^{\frac{1}{1+\theta}}$ in \eqref{a8} and get  that
$$
\norm{f}_{L^2(\Omega)} \leq 5^{\frac{s}{2}} C_d M (-\ln \epsilon)^{\frac{d}{1+\theta}} e^{-\frac{2}{\pi}(-\ln \epsilon)  \arctan (-\ln \epsilon)^{-\frac{1}{1+\theta}}}+ C (-\ln \epsilon)^{-\frac{k}{1+\theta}} \norm{f}_{H^k(\Omega)}.
$$
Since
$\arctan u =\int_0^u \frac{1}{1+v^2}  dv$ and since the function $v \mapsto\frac{1}{1+v^2}$ is decreasing on $[0,+\infty)$,
we obtain $\arctan u\geq \frac{u}{1+u^2}$ for all $u \ge 0$, and hence
\begin{eqnarray*}
\norm{f}_{L^2(\Omega)} &\leq &5^{\frac{s}{2}} C_d M (-\ln \epsilon)^{\frac{d}{1+\theta}} e^{-\frac{2}{\pi} \frac{(-\ln \epsilon)^{\frac{\theta}{1+\theta}}}{1+(-\ln \epsilon)^{-\frac{2}{1+\theta}}}} + C (-\ln \epsilon)^{-\frac{k}{1+\theta}} \norm{f}_{H^k(\Omega)}\\
& \le & (-\ln \epsilon)^{-\frac{k}{1+\theta}} \left( 5^{\frac{s}{2}} C_d M (-\ln \epsilon)^{\frac{d+k}{1+\theta}} e^{-\frac{2}{\pi} \frac{(-\ln \epsilon)^{\frac{2+\theta}{1+\theta}}}{1+(-\ln \epsilon)^{\frac{2}{1+\theta}}}} + C  \norm{f}_{H^k(\Omega)} \right).
\end{eqnarray*}
This entails that
\begin{equation}
\label{a9}
\norm{f}_{L^2(\Omega)} \le (-\ln \epsilon)^{-\frac{k}{1+\theta}} \left( 5^{\frac{s}{2}} C_d M (-\ln \epsilon)^{\frac{d+k}{1+\theta}} e^{-\frac{1}{\pi} (-\ln \epsilon)^{\frac{2+\theta}{1+\theta}}}+ C  \norm{f}_{H^k(\Omega)} \right),\ \epsilon \in (0,e^{-1}).
\end{equation}
Now, since $\lim_{\epsilon \downarrow 0} (-\ln \epsilon)^{\frac{d+k}{1+\theta}} e^{-\frac{1}{\pi} (-\ln \epsilon)^{\frac{2+\theta}{1+\theta}}}=0$, there exists $\epsilon_0>0$, depending only on $\Omega$, $d$, $k$, $s$, $\theta$ and $\frac{\norm{f}_{L^1(\Omega)}}{\norm{f}_{H^k(\Omega)}}$, such that we have
\begin{equation}
\label{a10}
\norm{f}_{L^2(\Omega)} \le C   \norm{f}_{H^k(\Omega)}(-\ln \epsilon)^{-\frac{k}{1+\theta}},\ \epsilon \in (0,\epsilon_0),
\end{equation}
which immediately yields \eqref{main1}.\\

\noindent {\it Second case: $d=s$}. Choosing $\delta=e^{-R^d}$ in \eqref{a7.1}, we get
\eqref{a8} where the constant $C_d$ is substituted for $2 C_d$. This leads to \eqref{main1} upon arguing as in the {\it First case}.\\

\noindent {\it Third case: $s>d$}. Taking $\delta=R^{-\frac{d}{s-d}}$ in \eqref{a7.1}, we obtain that
\begin{equation}
\label{a11}
\norm{f}_{L^2(\Omega)} \leq 5^{\frac{s}{2}} C_d M R^d \epsilon^{\frac{2}{\pi}\arctan R^{-1}}+
C R^{-k} \left( 1+ R^{\frac{k}{s-d}\left(s-d \left( 1+ \frac{d}{2k} \right)\right)} \right) \norm{f}_{H^k(\Omega)},
\end{equation}
where $2C_d$ is replaced by $C_d$.
\begin{enumerate}[i)]
\item If $s \leq  d+ \frac{d^2}{2k}$ then it is apparent that \eqref{a11} yields \eqref{a8}, from where we get \eqref{main1} by following the exact same path as in the {\it First case}.
\item If $s >  d+ \frac{d^2}{2k}$ then \eqref{a11} may be equivalently rewritten as
$$ 
\norm{f}_{L^2(\Omega)} \leq 5^{\frac{s}{2}} C_d M R^d \epsilon^{\frac{2}{\pi}\arctan R^{-1}}+ C R^{-\frac{d^2}{2(s-d)}}  \norm{f}_{H^k(\Omega)},
$$
which is the same estimate as \eqref{a8} where the power $k$ of the second occurence of $R$ is replaced by $\frac{d^2}{2(s-d)}$. 
Thus, by arguing in the same way as in the derivation of \eqref{a10} from \eqref{a8}, we obtain that
$$
\norm{f}_{L^2(\Omega)} \le C   \norm{f}_{H^k(\Omega)}(-\ln \epsilon)^{-\frac{d^2}{(1+\theta)(s-d)}},\ \epsilon \in (0,\epsilon_0),
$$
for some positive constant $\epsilon_0$ depending only on $\Omega$, $d$, $k$, $s$, $\theta$ and $\frac{\norm{f}_{L^1(\Omega)}}{\norm{f}_{H^k(\Omega)}}$, giving \eqref{main1}.
\end{enumerate}
\subsection{Proof of Theorem \ref{thm-inv2}}
We keep the notations of Section \ref{sec-pr_inv1}.
Assuming that $f =0$, we infer from \eqref{laplcasource} that
\bean
\label{eq001}
\abs{\hat{u_0}(i\omega \xi)} \leq (2\pi)^{-\frac{d}{2}}  \abs{\partial \Omega}^{\frac{1}{2}} 
e^{\kappa_\Omega \omega } \omega^{\frac{2-2\alpha}{\alpha}} \norm{\partial_\nu U(\omega, \cdot)}_{L^2(\partial \Omega)},\ \omega \in \R_+.
\eean  
Plugging \eqref{eq1} into \eqref{eq001} then yields
$$
\abs{\hat{u_0}(i\omega \xi)} \leq  (2\pi)^{-\frac{d}{2}}  \abs{\partial \Omega}^{\frac{1}{2}} \omega^{-1} e^{\kappa_\Omega \omega } 
\norm{\partial_\nu u}_{L^{\frac{2}{\alpha}}(\R_+;L^2(\partial \Omega))},\ \omega \in \R_+,
$$
and since $\omega^{-1} \le \left( \frac{2+\omega}{\omega} \right)^s$ for all $\omega \in \R_+$, where we recall that $s:= 1+ \lfloor \frac{2}{\alpha}\rfloor \in [3,+\infty)$, we end up getting that
\begin{equation}
\label{a12}
\left( \frac{\omega}{2+\omega} \right)^s e^{-\kappa_\Omega \omega} \abs{\hat{u_0}(i\omega \xi)} \leq  (2\pi)^{-\frac{d}{2}}  \abs{\partial \Omega}^{\frac{1}{2}}  
\norm{\partial_\nu u}_{L^{\frac{2}{\alpha}}(\R_+;L^2(\partial \Omega))},\ \omega \in [0,+\infty).
\end{equation}
The above estimate being similar to \eqref{a1} where $f$ replaced by $u_0$ and $c_0$ is equal to one, the desired result then follows from \eqref{a12} by arguing in the same way as in the derivation of Theorem \ref{thm-inv1} from \eqref{a1}.

\appendix
\section{Proof of Proposition \ref{pr-fwd}}
\label{sec-prfwd}
This section is devoted to the proof of Proposition \ref{pr-fwd}.
For this purpose we introduce the self-adjoint operator $A$ in $L^2(\Omega)$, acting as $-\Delta$ on his domain 
$D(A):=H^2(\Omega)\cap H^1_0(\Omega)$, that is to say the operator generated in $L^2(\Omega)$ by the closed quadratic form 
$a(u):=\int_\Omega \abs{\nabla u}^2 dx$, $u \in D(a):=H_0^1(\Omega)$. Since $H_0^1(\Omega)$ is compactly embedded in 
$L^2(\Omega)$, the operator $A$ has a compact resolvent and consequently a discrete spectrum. We denote by $(\lambda_k)_{k\in\N}$ the non-decreasing sequence of the eigenvalues of $A$, and by $\{ \phi_{k},\ k \in \N \}$ a set of eigenfunctions such that $A \phi_k = \lambda_k \phi_k$, which form an orthonormal basis in $L^2(\Omega)$.

With reference to \cite{KY1,SY}, the $L^2_{\mathrm{loc}}(\R_+; H^2(\Omega)\cap H^1_0(\Omega))$-solution $u$ to \eqref{directproblem} reads
\begin{equation}
\label{app-a1}
u(t)=S_0(t)u_0+\int_0^tS_1(t-s)fg(s)ds,\ t\in\R_+,
\end{equation}
where
\begin{equation}
\label{app-a2}
S_0(t)h:=\sum_{k=1}^{+\infty} E_{\alpha,1}(-\lambda_k t^\alpha)\langle h,\phi_k\rangle_{L^2(\Omega)},\ t\in\R_+,\ h \in L^2(\Omega),
\end{equation}
\begin{equation}
\label{app-a3}
S_1(t)h:=\sum_{k=1}^{+\infty} t^{\alpha-1}E_{\alpha,\alpha}(-\lambda_k t^\alpha) \langle h,\phi_k \rangle_{L^2(\Omega)},\ t\in\R_+,\ h \in L^2(\Omega),
\end{equation}
and $E_{\beta_1,\beta_2}$, for $(\beta_1,\beta_2)\in \R_+^2$, is the
the Mittag-Leffler function defined by
\begin{equation}
\label{app-a4}
E_{\beta_1,\beta_2}(z):=\sum_{k=0}^{+\infty} \frac{z^k}{\Gamma(\beta_1 k+\beta_2)},\ z\in\C.
\end{equation}
The rest of the proof is to show that the function $u$ defined by \eqref{app-a1}-\eqref{app-a4}, lies in $L^r(\R_+; H^{\frac{7}{4}}(\Omega))$ for all $r>\alpha^{-1}$, 
and that $u$ satisfies the energy estimate \eqref{reg.direct}. As will appear further, this essentially boils down to the following properties of 
the Mittag-Leffler functions defined by \eqref{app-a4}, which can be found in, e.g., \cite[Section 1.2.7 (pp. 34--35)]{P}.
\begin{lemma}
\label{lem:mlf}
Let $\beta_1 \in(0,2)$ and all $\beta_2 \in \R$. Then, there exists a constant $c= c(\beta_1, \beta_2)>0$ such that
\begin{equation}  
\label{app-a5}
\abs{E_{\beta_1,\beta_2}(\tau)} \leq \frac{c}{1+\abs{\tau}},\ \tau\in(-\infty,0],
\end{equation}
and it holds true for all $N \in \N$ that
\begin{equation}
\label{app-a6}
E_{\beta_1,\beta_2}(\tau)  = -\sum_{k=1}^N \frac{\tau^{-k}}{\Gamma(\beta_2-\beta_1 k)} + \underset{\tau\to-\infty}{\mathcal{O}}(\abs{\tau}^{-N-1}).
\end{equation}
\end{lemma}
Here and in the remaining part of this appendix, we follow the convention used in \cite{P} by setting
\begin{equation}
\label{ga}
\frac{1}{\Gamma(m)}:=0,\ m \in \Z\setminus\N := \{ \ldots,-2,-1,0 \}.
\end{equation}

We recall that for all $\gamma>0$, 
$D(A^\gamma) =\{ v \in L^2(\Omega),\ \sum_{k=1}^{+\infty} \lambda_k^{2 \gamma} \abs{\langle v , \varphi_k \rangle_{L^2(\Omega)}}^2 < \infty \}$ is a Hilbert space with the norm
$$ \norm{v}_{D(A^\gamma)}:= \left( \sum_{k=1}^{+\infty} \lambda_k^{2 \gamma} \abs{\langle v , \varphi_k \rangle_{L^2(\Omega)}}^2\right)^{\frac{1}{2}},\ v \in D(A^\gamma), $$
and that $D(A^\gamma)\subset H^{2\gamma}(\Omega)$, the injection being continuous. 
In light of this we write
$\norm{S_0(t)u_0}_{H^2(\Omega)} \leq C \norm{S_0(t)u_0}_{D(A)}$, where, from now on, $C$ denotes a generic positive constant depending only on $\Omega$ and $\alpha$, which may change from line to line,
and then deduce from \eqref{app-a5} that
\begin{eqnarray*}
\norm{S_0(t)u_0}_{H^2(\Omega)}^2 & \leq & C \sum_{k=1}^{+\infty} \lambda_k^2 \abs{E_{\alpha,1}(-\lambda_k t^\alpha)}^2 \abs{\langle u_0,\phi_k\rangle_{L^2(\Omega)}}^2\\
&\leq & C\sum_{k=1}^{+\infty} \frac{\lambda_k^2}{(1+\lambda_kt^\alpha)^2} \abs{\langle u_0,\phi_k \rangle_{L^2(\Omega)}}^2\\
&\leq & \frac{C}{(1+\lambda_1t^\alpha)^2}\sum_{k=1}^{+\infty} \lambda_k^2 \abs{\langle u_0,\phi_k \rangle_{L^2(\Omega)}}^2.
\end{eqnarray*}
Thus, we have $\norm{S_0(t)u_0}_{H^2(\Omega)} \leq \frac{C}{1+\lambda_1t^\alpha} \norm{u_0}_{D(A)}$, and consequently
$$
\norm{S_0(t)u_0}_{H^2(\Omega)}\leq \frac{C\norm{u_0}_{D(A)}}{1+t^\alpha},\ t \in \R_+.
$$
For $r >\alpha^{-1}$, this entails that $t \mapsto S_0(t)u_0 \in L^r \left( \R_+;H^{\frac{7}{4}}(\Omega) \right)$, with
\begin{equation}
\label{p1a}
\norm{S_0(\cdot)u_0}_{L^r \left( \R_+;H^{\frac{7}{4}}(\Omega) \right)} \leq C \norm{u_0}_{H^2(\Omega)}.
\end{equation}
Similarly, since $\norm{S_1(t)f}_{H^{\frac{7}{4}}(\Omega)} \leq C \norm{S_1(t)f}_{D(A^{\frac{7}{8}})}$ for all $t\in\R_+$, we have
\begin{eqnarray*}
\norm{S_1(t)f}_{H^{\frac{7}{4}}(\Omega)}^2 
&\leq & C t^{2(\alpha-1)} \sum_{k=1}^{+\infty} \lambda_k^{\frac{7}{4}} \abs{E_{\alpha,\alpha}(-\lambda_k t^\alpha)}^2
\abs{\langle f,\phi_k\rangle_{L^2(\Omega)}}^2\\
&\leq & Ct^{2(\alpha-1)}\sum_{k=1}^{+\infty} \frac{\lambda_k^{\frac{7}{4}}}{(1+\lambda_kt^\alpha)^2}\abs{\langle f,\phi_k\rangle_{L^2(\Omega)}}^2.
\end{eqnarray*}
Taking into account that
$
\frac{\lambda_k^{\frac{7}{4}}}{(1+\lambda_kt^\alpha)^2} = t^{-\frac{7\alpha}{4}} \left( \frac{\lambda_k t^{\alpha}}{1+ \lambda_k t^\alpha} \right)^{\frac{7}{4}} \frac{1}{(1+\lambda_kt^\alpha)^{\frac{1}{4}}} \leq t^{-\frac{7\alpha}{4}}$ for all $k \in \N$, we get from the above estimate that
$\norm{S_1(t)f}_{H^{\frac{7}{4}}(\Omega)}^2 \leq Ct^{2(\alpha-1)}t^{-{\frac{7}{4}}\alpha}\sum_{k=1}^{+\infty} \abs{\langle f,\phi_k \rangle_{L^2(\Omega)}}^2$, which entails that
\begin{equation}
\label{p1b}
\norm{S_1(t)f}_{H^{\frac{7}{4}}(\Omega)} \leq
C t^{2 \left( \frac{\alpha}{8}-1\right)} \norm{f}_{L^2(\Omega)},\ t \in \R_+.
\end{equation}
Further, applying \eqref{app-a6} with $\beta_1=\beta_2=\alpha$ and $N=2$, we have 
$E_{\alpha,\alpha}(\tau) = \frac{\tau^{-2}}{\Gamma(-\alpha)} + \underset{\tau\to-\infty}{\mathcal{O}}(\abs{\tau}^{-3})$,
by virtue of \eqref{ga}. Thus, there exists a positive constant $C$, depending only on $\alpha$, such that the following estimate
$$\abs{t^{\alpha-1} E_{\alpha,\alpha}(-\lambda_k t^\alpha)} \leq C t^{\alpha-1}(\lambda_kt^\alpha)^{-2}
\leq C\lambda_k^{-2}t^{-(\alpha+1)},\ t \in [1,+\infty),$$
holds uniformly in $k \in \N$. As a consequence we have
\begin{eqnarray*}
\norm{S_1(t)f}_{H^2(\Omega)}^2&\leq & C \norm{S_1(t)f}_{D(A)}^2 \\
& \leq & C\sum_{k=1}^{+\infty} \lambda_k^2 \abs{t^{\alpha-1}E_{\alpha,\alpha}(-\lambda_k t^\alpha)}^2 \abs{\langle f,\phi_k\rangle_{L^2(\Omega)}}^2\\
&\leq & C \sum_{k=1}^{+\infty} \lambda_k^{-2} t^{-2(\alpha+1)} \abs{\langle f,\phi_k\rangle_{L^2(\Omega)}}^2,
\end{eqnarray*}
and hence $\norm{S_1(t)f}_{H^2(\Omega)} \leq C t^{-(\alpha+1)} \norm{f}_{L^2(\Omega)}$ for all $t \in [1,+\infty)$, where we used that $\lambda_k^{-2} \leq \lambda_1^{-2}$ for all $k \in \N$. From this and \eqref{p1b} it then follows that
\begin{equation}
\label{p1c}
\norm{S_1(t)f}_{H^{\frac{7}{4}}(\Omega)} \leq C \left(t^{\frac{\alpha}{8}-1}\mathds{1}_{(0,1)}(t)+t^{-(\alpha+1)}\mathds{1}_{[1,+\infty)}(t)\right),\ t\in \R_+,
\end{equation}
where the notation $\mathds{1}_{I}$ stands for the characteristic function of any subinterval $I \subset \R$. 
Putting \eqref{p1b} and \eqref{p1c} together, we obtain that $S_1(t) f \in L^1 \left(\R_+;H^{\frac{7}{4}}(\Omega) \right)$, with
\begin{equation}
\label{p1d}
\norm{S_1(\cdot)f}_{L^1 \left(\R_+;H^{\frac{7}{4}}(\Omega) \right)} \leq C \norm{f}_{L^2(\Omega)}.
\end{equation}
Moreover, we have
\begin{eqnarray*}
\norm{u(t)}_{H^{\frac{7}{4}}(\Omega)} & \leq\ & \norm{S_0(t)u_0}_{H^{\frac{7}{4}}(\Omega)}+\int_0^t \norm{S_1(t-s)f}_{H^{\frac{7}{4}}(\Omega)} \abs{g(s)}ds\\
&\leq& \norm{S_0(t)u_0}_{H^{\frac{7}{4}}(\Omega)} +\left( \norm{S_1(\cdot)f}_{H^{\frac{7}{4}}(\Omega)} \mathds{1}_{\R_+}\right) \star \left(\abs{g} \mathds{1}_{\R_+}\right)(t),\ t \in \R_+,
\end{eqnarray*}
from \eqref{app-a1}, and hence
\begin{eqnarray*}
\norm{u}_{L^r(\R_+;H^{\frac{7}{4}}(\Omega))} & \leq & \norm{S_0(\cdot)u_0}_{L^r(\R_+;H^{\frac{7}{4}}(\Omega))}+\norm{S_1(\cdot)f}_{L^1(\R_+;H^{\frac{7}{4}}(\Omega))} \norm{g}_{L^r(\R_+)}\\
&\leq & C \left(\norm{u_0}_{H^2(\Omega)}+\norm{g}_{L^r(\R_+)} \norm{f}_{L^2(\Omega)}\right),
\end{eqnarray*}
upon combining Young's convolution inequality with \eqref{p1a} and \eqref{p1d}.
This completes the proof of Proposition \ref{pr-fwd}.

\section{Unique continuation}
\label{sec-uc}

Let $s \in \N$, let $\xi \in \mathbb S^{d-1}$ and let $\phi \in L^1(\Omega)$. Put $\cH:=\{ z \in \C,\ \re z > -1 \}$. With reference to \eqref{def-Fourier} we introduce
\begin{equation}
\label{FFF}
F(z):=  \left( \frac{z}{2+z} \right)^s e^{-\kappa_\Omega z} \hat{\phi}(iz\xi),\ z \in \overline{\cH} = \{ z \in \C,\ \re z \ge -1 \}, 
\end{equation}
where we recall that $\hat{\phi}$ is the Fourier transform of $\phi$ and $\kappa_\Omega=\sup_{x \in \Omega} \abs{x}$ is the Euclidean diameter of $\Omega$. 

Evidently, $F$ is holomorphic in the half-plane $\cH$ and continuous on $\partial \cH = \{ z \in \C,\ \re z=-1\}$. Moreover, we have 
\begin{equation}
\label{app-b1}
\abs{F(z)} \leq (2 \pi)^{-d \slash 2} e^{2 \kappa_\Omega} \norm{\phi}_{L^1(\Omega)},\ z \in \overline{\cH},
\end{equation}
directly from \eqref{def-Fourier}. With reference to \eqref{app-b1} we set for further use
\begin{equation}
\label{app-b2}
M := 1+ (2 \pi)^{-d \slash 2} e^{2 \kappa_\Omega} \norm{\phi}_{L^1(\Omega)} \in [1,+\infty).
\end{equation}
As we aim to compute a suitable upper bound of $F$ in the quadrant $Q := \{ z \in \C,\ \re z >-1\ \mbox{and}\ \im z <0 \}$, we introduce
\begin{equation}
\label{def-w}
w(z):=\frac{2}{\pi} \left( \frac{\pi}{2}+ \arg(z+1) \right),\ z \in \overline{Q} \setminus \{-1\} = \{ z \in \C,\ \re z \ge -1\ \mbox{and}\ \im z \le 0 \} \setminus \{ -1 \}
\end{equation}
and we recall the following result, which is borrowed from \cite{AGT}.
\begin{proposition} 
\label{pr-B1}
The function $w$ defined in \eqref{def-w} is a harmonic measure and is the unique solution to the system
\begin{equation}
\label{app-b3}
\left\{ \begin{array}{ll}
\Delta w(z) = 0, & z \in Q, \\
w(t) =1, & t \in (-1,+\infty), \\
w(-1-it)=0, & t \in (0,+\infty).
\end{array} \right. 
\end{equation}
\end{proposition}
\begin{proof} To make the present paper self-contained and for the convenience of the reader we include the proof of this result. 
First, we notice that
$w(t)=\frac{2}{\pi} \left( \frac{\pi}{2}+ \arg(t+1) \right)=\frac{2}{\pi} \left( \frac{\pi}{2} + 0 \right)=1$ for all $t\in(-1,+\infty)$ and that
$w(-1-it)=\frac{2}{\pi} \left( \frac{\pi}{2}+ \arg(-it) \right)=\frac{2}{\pi} \left( \frac{\pi}{2} - \frac{\pi}{2} \right)=0$ for $t\in(0,+\infty)$.
Next, in order to show that $\Delta w=0$ in $Q$, it is enough to we rewrite $w$ as
$$w(z)=\frac{2}{\pi} \left(\frac{\pi}{2}+ \im \log(z+1) \right),\ z\in Q,$$
where $\log$ denotes the complex logarithmic function $\log$, defined in $\{ z \in\C,\ -i z \notin[0,+\infty) \}$. The desired result then follows from the holomorphicity of $z \mapsto \log(z+1)$ in the quadrant $Q$. 
\end{proof}

\begin{theorem} 
\label{thmUC}
Let $F$ be defined by \eqref{FFF}, let $M$ be given by \eqref{app-b2} and let $w$ be the same as in \eqref{def-w}. Then, we have
\begin{equation}
\label{ttt1a}
\abs{F(z)} \leq  M m^{w(z)},\ z\in \overline{Q} \setminus \{-1\},
\end{equation}
where  $m:=\sup_{t\in[-1,+\infty)} \abs{F(t)}$.
\end{theorem}
\begin{proof}
By combining the estimate $\abs{F(t)} \leq m$ for all $t \in (-1,+\infty)$, with the basic identity $m=M^{1-1} m^1$ and the second line of \eqref{app-b3}, we obtain that 
\begin{equation}
\label{app-b4}
\abs{F(t)} \leq M^{1-w(t)} m^{w(t)},\ t \in (-1,+\infty).
\end{equation}
Further we have
\begin{equation}
\label{ttt1b}
\abs{F(z)} \leq M,\ z\in \overline{Q},
\end{equation}
from \eqref{app-b1}-\eqref{app-b2}, hence $\abs{F(-1-it)} \leq M$ for all $t \in (0,+\infty)$. From this, the identity $M=M^{1-0} m^0$ and the third line of \eqref{app-b3}, it then follows that
\begin{equation}
\label{app-b5}
\abs{F(-1-it)} \leq M^{1-w(-1-it)} m^{w(-1-it)},\ t \in (0,+\infty),
\end{equation}
Summing up \eqref{app-b4} and \eqref{app-b5}, we have
\begin{equation}
\label{ttt1c}
\abs{F(z)}\leq M^{1-w(z)}m^{w(z)},\ z \in \partial Q \setminus \{- 1 \}.
\end{equation}
Since $F$ is holomorphic in $Q$ and $w$ is a harmonic measure of $Q$ by Proposition \ref{pr-B1}, \eqref{ttt1b}, \eqref{ttt1c} and the Two-constants theorem (see, e.g., \cite[Chap. III, Section 2.1]{Ne} or \cite{TT}) yield
$$\abs{F(z)} \leq M^{1-w(z)} m^{w(z)},\ z\in Q.$$
Thus, by continuity of $F$ on $\partial Q$ and $w$ on $\partial Q \setminus \{ -1 \}$, we obtain that
$$\abs{F(z)} \leq M^{1-w(z)} m^{w(z)},\ z\in \overline{Q} \setminus \{ -1 \}.$$
Finally, \eqref{ttt1a} follows readily from this and the inequality $M \geq 1$ arising from \eqref{app-b2}.
\end{proof}



\end{document}